\documentclass{amsart}

\usepackage{amsmath,amssymb,amsthm}
\usepackage{a4wide}
\usepackage{hyperref}
\usepackage{multicol}

\usepackage{graphicx}
\usepackage{xypic}
\entrymodifiers={+!!<0pt,\fontdimen22\textfont2>}
\usepackage[all]{xy}

\newtheoremstyle{myremark} 
    {7pt}                    
    {7pt}                    
    {}  	                 
    {}                           
    {\bf}       	         
    {.}                          
    {.5em}                       
    {}  

\theoremstyle{plain}
\newtheorem{lemma}{Lemma}
\newtheorem{theorem}[lemma]{Theorem}

\theoremstyle{myremark}

\newcommand{\nicec}{\mathcal{C}}

\newcommand{\zet}{\mathbb{Z}}

\newcommand{\ind}{\mathrm{Ind}}

\newcommand{\redhom}{H}

\newcommand{\cl}{\mathrm{Cl}}
\newcommand{\rptwo}{\mathbb{R}P^2}

\newcommand{\lk}{\mathrm{lk}}

\newcommand{\longversion}[1]{}

\begin{document}
\title{Small flag complexes with torsion}

\author[Micha{\l} Adamaszek]{Micha{\l} Adamaszek}
\address{Mathematics Institute and DIMAP,
      \newline University of Warwick, Coventry, CV4 7AL, UK}
\email{aszek@mimuw.edu.pl}
\thanks{Research supported by the Centre for Discrete
        Mathematics and its Applications (DIMAP), EPSRC award EP/D063191/1.}


\keywords{Clique complex, order complex, homology, torsion, minimal model}

\begin{abstract}
We classify flag complexes on at most $12$ vertices with torsion in the first homology group. The result is moderately computer-aided.

As a consequence we confirm a folklore conjecture that the smallest poset whose order complex is homotopy equivalent to the real projective plane (and also the smallest poset with torsion in the first homology group) has exactly $13$ elements.
\end{abstract}
\maketitle

There is a well-known $6$-vertex simplicial triangulation of the real projective space $\rptwo$. It is smallest in terms of the number of vertices and it is also the minimal simplicial complex with torsion in the first homology group. In this note we consider analogous minimization questions in the classes of flag complexes and order complexes, both of which are widely used combinatorial models of topological spaces.

If $G$ is a simple, undirected graph, then the \emph{clique complex} $\cl(G)$ of $G$ is the simplicial complex whose vertices are the vertices of $G$ and whose faces correspond to the cliques in $G$. Clique complexes are also known by the name \emph{flag complexes}. We will show the following fact.

\begin{theorem}
\label{thm:a}
We have the following classification:
\begin{itemize}
\item[a)] If $G$ is a graph with at most $10$ vertices then $\redhom_1(\cl(G);\zet)$ is torsion-free.
\item[b)] There are exactly four graphs $K_1$, $K_2$, $K_3$, $K_4$ with $11$ vertices for which $\redhom_1(\cl(G);\zet)$ has torsion.
\item[c)] There exist $363$ graphs $L_1,\ldots, L_{363}$ with $12$ vertices and with the following property. If $G$ is any $12$-vertex graph for which $\redhom_1(\cl(G);\zet)$ has torsion then either $G$ is one of the $L_i$ or $G\setminus v$ is one of the $K_i$ for some vertex $v$ of $G$.
\end{itemize}
\end{theorem}

Parts a) and b) of the above theorem were also proved in \cite{Katzman}. Our main effort is in proving part c), but the method we use also verifies a) and b). For a list of the graphs $K_i$, $L_i$ see Section~\ref{section:questions}.

Next, if $P$ is a poset, then the \emph{order complex} $\Delta(P)$ of $P$ is the simplicial complex whose vertices are the elements of $P$ and whose faces correspond to chains in $P$. This is the standard construction of the classifying space of $P$. 

Note that $\Delta(P)$ is the clique complex of the \emph{comparability graph} of $P$ (this graph has vertex set $P$ and an edge between every two comparable elements). By using the classification given by Theorem~\ref{thm:a} we will obtain the following.

\begin{theorem}
\label{thm:b}
If $P$ is a poset with at most $12$ elements then the group $\redhom_1(\Delta(P);\zet)$ is torsion-free.
\end{theorem}

This confirms a conjecture stated in \cite[Sect.4]{HVB}, \cite[Conj.5.4]{Weng} or \cite[Ex.7.1.1]{BarmakBook} that the smallest poset whose classifying space is homotopy equivalent to $\rptwo$ has $13$ elements. A poset $P$ with exactly $13$ elements and with $\Delta(P)$ homeomorphic to $\rptwo$ is known (see the same references). For the relation between the minimal poset and the minimal model in $T_0$-spaces see \cite{BarMin,May}.

The obvious way to prove Theorem~\ref{thm:a} would be to compute $\redhom_1(\cl(G);\zet)$ for all graphs $G$ on at most $12$ vertices. However, there is around $1.6\cdot 10^{11}$ such graphs \cite[A000088]{Oeis}, which makes a direct check infeasible. An alternative approach to Theorem~\ref{thm:b} is to go through all posets with at most $12$ elements. There is around $10^9$ of them \cite[A000112]{Oeis}, but this time the problem lies in the non-availability of good software for generation of posets (at least to the author's knowledge). Our approach is to reduce the search space of graphs so that in the end homology must be computed only for less than $10^8$ cases.

\section{Enumeration of graphs with torsion}

We will use some standard notation. If $v$ is a vertex of a graph $G$ then $N_G(v)$ is the set of neighbors of $v$ in $G$. We write $\lk_G v = G[N_G(v)]$ for the subgraph of $G$ induced by the neighborhood of $v$. Note that this notation coincides with the usual notion of link for simplicial complexes, i.e. we have $\lk_{\cl(G)} v = \cl(\lk_G v)$. The degree of a vertex $v$ is $\deg_G(v)=|N_G(v)|$. The complement $\overline{G}$ is the graph with vertex set $V(G)$ whose edges are the non-edges of $G$. The independence complex $\ind(G)$ of $G$ is defined as $\ind(G)=\cl(\overline{G})$.

In order to prove Theorem~\ref{thm:a} it suffices to characterize those graphs with torsion in first homology for which the removal of any vertex yields a graph without torsion. This motivates the next few definitions.

Let $G$ be a graph with $n$ vertices. We will say that $G$ \emph{has cyclic links} if for every vertex $v$ of $G$ the group $\redhom_1(\cl(\lk_G v);\zet)$ is nontrivial. We say that $G$ is an \emph{$\redhom_1$-torsion} graph if the group $\redhom_1(\cl(G);\zet)$ has torsion. If $G$ is $\redhom_1$-torsion then we will say $G$ is \emph{irreducible} if for every vertex $v$ of $G$ the graph $G\setminus v$ is not $\redhom_1$-torsion. Finally, we say $G$ is \emph{tame} if every vertex $v$ satisfies $4\leq\deg_G(v)\leq n-4$.

The next lemma rephrases known results about independence complexes.
\begin{lemma}
\label{lemma:ind}
If $G$ is a graph with $n$ vertices and $v$ is a vertex with $\deg_G(v)\geq n-3$ then $\cl(G)$ is homotopy equivalent to the suspension of some space.
\end{lemma}
\begin{proof}
Let $H=\overline{G}$. We have $\deg_H(v)\leq 2$ and $\cl(G)=\ind(H)$.

If $\deg_H(v)=0$ then $\ind(H)$ is a cone, hence a contractible space. If $\deg_{H}(v)=1$ or $\deg_{H}(v)=2$ and the two vertices of $N_{H}(v)$ are adjacent then the result follows from \cite[Lemma 2.5]{Engstrom}. If $\deg_{H}(v)=2$ and the two vertices of $N_{H}(v)$ are non-adjacent then we use \cite[Thm. 3.4]{Barmak}.
\end{proof}

The key to (fairly) efficient enumeration of irreducible $\redhom_1$-torsion graphs is the following observation.

\begin{lemma}
\label{lemma:key}
If $G$ is an irreducible $\redhom_1$-torsion graph then $G$ is connected, tame and it has cyclic links.
\end{lemma}
\begin{proof}
It is clear that $G$ is connected. We start by proving that $G$ has cyclic links. Let $v$ be any vertex of $G$. We have a cofibre sequence
$$\cl(\lk_G v) \to \cl(G\setminus v) \to \cl(G)$$
and hence a long exact sequence of homology groups (with $\zet$ coefficients):
$$\cdots\to\redhom_1(\cl(\lk_G v))\to\redhom_1(\cl(G\setminus v))\to\redhom_1(\cl(G))\to\redhom_0(\cl(\lk_G v))\to\cdots.$$
The conclusion $\redhom_1(\cl(\lk_G v))\neq 0$ follows by a standard exact sequence argument from the fact that $\redhom_1(\cl(G))$ has torsion while $\redhom_1(\cl(G\setminus v))$ and $\redhom_0(\cl(\lk_G v))$ are torsion-free.

Next we prove $G$ is tame. Since $\redhom_1(\cl(H);\zet)=0$ for all at most $3$-vertex graphs $H$, the cyclic links condition means that for every vertex $v$ of $G$ we have $|N_G(v)|\geq 4$, i.e. $\deg_G(v)\geq 4$. The other inequality follows from Lemma~\ref{lemma:ind} since the first homology group of a suspension is torsion-free.
\end{proof}

The next lemma records the computer-assisted part of the argument.

\begin{lemma}
\label{lemma:computer}
If $G$ is an irreducible $\redhom_1$-torsion graph with at most $12$ vertices, then $G$ is one of the graphs $K_1,\ldots,K_4$, $L_1,\ldots,L_{363}$ appearing in Theorem~\ref{thm:a}.
\end{lemma}
\begin{proof}
By Lemma~\ref{lemma:key} all irreducible $\redhom_1$-torsion graphs can be found among connected, tame graphs with cyclic links. Let $n\leq 12$ be the number of vertices we are considering. If $n\leq 7$ then there are no tame graphs. For each $8\leq n\leq 12$ we generate all $n$-vertex, connected graphs, pick the tame ones and among those pick the ones with cyclic links. In the resulting set of graphs we then check for torsion in the first homology of the clique complex, and finally, in the case $n=12$, we  eliminate the graphs which are reducible.

The numbers of graphs which arise at the consecutive steps of this reduction are shown in Table~\ref{table:count}. More specific implementation details are described in Section~\ref{section:details}.
\end{proof}

Lemma~\ref{lemma:computer} clearly implies Theorem~\ref{thm:a}.

\bigskip
We now proceed with the proof of Theorem~\ref{thm:b}.

\begin{lemma}
\label{lemma:5cycle}
Each of the graphs $K_i$, $L_i$ of Theorem~\ref{thm:a} contains an induced $5$-cycle.
\end{lemma}
\begin{proof}
This is an immediate brute-force computer check.
\end{proof}

\begin{proof}[Proof of Theorem~\ref{thm:b}]
Suppose, on the contrary, that $P$ is a poset with at most $12$ vertices and with torsion in $\redhom_1(\Delta(P);\zet)$. Let $G$ be the comparability graph of $P$, so that $\Delta(P)=\cl(G)$. By Theorem~\ref{thm:a} the graph $G$ contains, as an induced subgraph, one of $K_i$ or $L_i$ and therefore, by Lemma~\ref{lemma:5cycle}, it also contains an induced $5$-cycle. However, the comparability graph of a poset cannot have an induced cycle of odd length greater than three, and this contradiction ends the proof.
\end{proof}

\begin{table}
\begin{tabular}{l|r|r|r|r|r}
$n$ & \parbox{2.8cm}{connected graphs\newline \cite[A001349]{Oeis}} & \parbox{2.5cm}{connected, tame}  &  \parbox{2.5cm}{connected, tame\newline with cyclic links} & \parbox{2.5cm}{connected, tame\newline with cyclic links\newline and $\redhom_1$-torsion} & \parbox{2cm}{irreducible\newline $\redhom_1$-torsion}\\ \hline
8 & 11\,117 & 6 & 0 & 0 & 0\\
9 & 261\,080 & 634 & 2 & 0 & 0\\
10 & 11\,716\,571 & 194\,917 & 492 & 0 & 0\\
11 & 1\,006\,700\,565 & 64\,434\,518 & 207\,839 & 4 & 4\\
12 & 164\,059\,830\,476 & 26\,169\,627\,695 & 93\,453\,159 & 394 & 363\\
\end{tabular}

\caption{Various graph classes appearing in the consecutive steps of the computation.\label{table:count}}
\end{table}

\section{Implementation outline}
\label{section:details}
We give an outline of the calculation for $12$-vertex graphs. 

Let $\nicec_8$ be the set of all graphs $H$ with exactly $8$ vertices (not necessarily connected) and with $\redhom_1(\cl(H);\zet)\neq 0$. This set can be computed by a brute-force algorithm that checks all $12346$ of the $8$-vertex graphs. We have $|\nicec_8|=7702$.

If $G$ is a graph and $k$ is an integer let $G+k$ denote $G$ with extra $k$ isolated vertices. If $\nicec$ is a class of graphs then $\nicec+k=\{G+k~:~G \in \nicec\}$. When $G$ is a graph, and $v$ is its vertex, let $G_v$ denote the graph with the same vertices as $G$ and with the edge set $\{xy~:~x,y\in N_G(v),\ xy\in E(G)\}$. In other words, $G_v=\lk_Gv + (|V(G)|-\deg_G(v))$.

In the first phase we use the program \texttt{geng} from the \texttt{nauty} package \cite{Nauty} to generate all connected, tame $12$-vertex graphs (\texttt{geng -c -d4 -D8}). From this set we need to choose graphs with cyclic links. This condition is easily verified as follows. If $G$ is a graph, and $v$ is its vertex, then the graph $G_v$ is easily computable in the internal representation of \texttt{nauty} using quick bit operations. We have that $\redhom_1(\cl(G_v);\zet)\neq 0$ if and only if $\redhom_1(\cl(\lk_G v);\zet)\neq 0$. Since $G$ is tame we have $G_v=H+4$ for some $8$-vertex graph $H$. From this we conclude that it suffices to check that $G_v\in \nicec_8+4$. We now use the capability of \texttt{nauty} to compute canonical representations of graphs, which have the property that two graphs are isomorphic if and only if their canonical representations are equal. We precompute the canonical representations of graphs in $\nicec_8+4$, sort them, and for every considered vertex $v$ of a graph $G$ we 
binary-search if the canonical representation of $G_v$ is present in that list.

The first phase required in total approx. 184 hours of processor time and yielded approx. $9\cdot 10^7$ graphs (Table~\ref{table:count}). This quantity is surprisingly close to what is predicted by the appealing (but not correct) heuristic which assumes that all vertex links in $G$ are independent random graphs, with each isomorphism class equally likely. Then the probability of $G$ having cyclic links would be $(\frac{7702}{12346})^{12}\approx 0.003474$, giving the expected number of such graphs (among tame graphs) as approximately $26\cdot 10^9\cdot 0.003474\approx 9\cdot 10^7$.

In the second phase we use the \texttt{chomp} program \cite{Chomp} to compute $\redhom_1(\cl(G);\zet)$ for all the graphs obtained in the first phase. In fact, it is faster to check $\redhom_1$ of the two-dimensional simplicial complex whose maximal faces are the triangles in $G$. This does not influence the existence of torsion. This phase required approx. 30 hours and produced $394$ graphs. At the end we eliminate the graphs which still contain one of $K_i$ as an induced subgraph and this leaves the final $363$ graphs $L_i$.

Note that at the same speed the brute-force check by \texttt{chomp} of the homology of all tame, connected, $12$-vertex graphs would take approx. 350 days.

\section{Conclusion}
\label{section:questions}

It seems natural to expect that the minimal simplicial complex (flag complex / order complex) with torsion in homology would be somehow related to the two-dimensional real projective space $\rptwo$. Indeed, one checks directly with \texttt{polymake} \cite{Polymake} and \texttt{chomp} \cite{Chomp} that
\begin{itemize}
\item Two of the complexes $\cl(K_i)$ are homeomorphic to $\rptwo$ and the remaining two collapse to $\rptwo$.
\item Among the complexes $\cl(L_i)$ there are
\begin{itemize}
\item $14$ spaces homeomorphic to $\rptwo$,
\item $344$ complexes which collapse to $\rptwo$,
\item $5$ spaces homotopy equivalent to $\rptwo\vee S^1$.
\end{itemize}
\end{itemize}

\medskip
A full list of the graphs $K_i$ and $L_i$ in \texttt{nauty} \cite{Nauty} format is embedded in the \texttt{.tex} source of this paper, available from the \texttt{arXiv} repository.

\subsection*{Acknowledgement} Thanks to Jonathan Barmak for discussions on this topic.


\end{document}